\documentclass[final,3p]{elsarticle}
\usepackage{geometry}
\usepackage{mathrsfs,enumerate,amssymb,amsmath,bm,color,lineno}
\usepackage[all]{xy}
\usepackage{latexsym}
\usepackage{amsthm,a4wide}
\usepackage{amscd,verbatim}
\usepackage{graphicx}
\usepackage[colorlinks=true]{hyperref}
\theoremstyle{plain}
\newtheorem{theorem}{Theorem}
\newtheorem{lemma}{Lemma}
\newtheorem{proposition}{Proposition}
\newtheorem{corollary}{Corollary}

\theoremstyle{definition}
\newtheorem{definition}{Definition}

\newtheorem{claim}{Claim}
\numberwithin{equation}{section}

\newtheorem{remark}{Remark}

\usepackage{pifont}
\newcommand{\hollowstar}{\text{\ding{78}}}

\newcommand{\chinastar}{\text{\ding{75}}}

\journal{Information Sciences}

\begin{document}

\begin{frontmatter}
\title{Revisiting type-2 triangular norms on normal convex fuzzy truth values\tnoteref{mytitlenote}}
\tnotetext[mytitlenote]{This work was supported by the Natural Science
Foundation of Sichuan Province (No. 2022NSFSC1821), the National Natural Science Foundation of China
(No. 11601449), and the Key Natural Science Foundation of Universities in Guangdong Province
(No. 2019KZDXM027).}

\author[a0,a1]{Xinxing Wu\corref{mycorrespondingauthor}}
\cortext[mycorrespondingauthor]{Corresponding author}
\address[a0]{School of Mathematics and Statistics, Guizhou University of Finance and Economics, Guiyang,
Guizhou 550025, China}
\address[a1]{Zhuhai College of Science and Technology, Zhuhai, Guangdong 519041, China}
\ead{wuxinxing5201314@163.com}

\author[a2]{Zhiyi Zhu}
\address[a2]{School of Sciences, Southwest Petroleum University, Chengdu, Sichuan 610500, China}
\ead{zhuzhiyi2019@163.com}

\author[a3]{Guanrong Chen}
\address[a3]{Department of Electrical Engineering, City University of Hong Kong,
Hong Kong SAR, China}
\ead{gchen@ee.cityu.edu.hk}

\begin{abstract}
This paper studies t-norms on the space $\mathbf{L}$ of all normal and convex fuzzy truth values.
We first prove that the only non-convolution form type-2 t-norm constructed by Wu et al. satisfies
the distributivity law for meet-convolution and show that t-norm in the sense of Walker and Walker
is strictly stronger than t$_{r}$-norm on $\mathbf{L}$, which is strictly stronger than t-norm on
$\mathbf{L}$. Furthermore, we characterize some restrictive axioms of t$_{r}$-norms for convolution
operations on $\mathbf{L}$ and obtain some necessary conditions for t$_{r}$-(co)norm convolution
operations on $\mathbf{L}$.
\end{abstract}

\end{frontmatter}

\section{Introduction}
Fuzzy truth values are truth values of type-2 fuzzy sets introduced by Zadeh~\cite{Z1975}, which map the unit
interval $I=[0, 1]$ to itself. It is well known that the numerical truth values form a Kleene algebra under the
minimum $\wedge$ and maximum $\vee$ operations~\cite{M1981}, and interval truth values, triangle truth values,
and normal and convex fuzzy truth values form De Morgan algebras under meet-convolution $\sqcap$ and
join-convolution $\sqcup$ \cite{M1994}. The systematic research for the algebraic structures of fuzzy truth
values began with the work of Walker and Walker \cite{WW2005} in 2005. Since then, it has attracted many scholars
and been extensively studied \cite{EM1999,HWW2016,HWW2008,HWW2010,HWW2015,T2009,WW2006b,WW2008b,WW2008}. Another
interesting research for fuzzy truth values focuses on operations on it, including triangular (co)norms
\cite{KM2002,Li2015,LW2020-2,S2009,WW2006,Zhang2016}, t$_{r}$-(co)norms \cite{HCT2015,WCW2022,WC-TFS},
uninorms and nullnorms \cite{WL2020,X2018}, Zadeh's extension \cite{LW2020}, convolution operations 
\cite{DBD2018,ZH2020-2,ZH2020}, negation \cite{CTH2020,TCH2018}, and aggregation \cite{T2014,W2015,W2015-2}.

Recently, Hern\'{a}ndez et al.~\cite{HCT2015} introduced two binary operations $\widetilde{\Delta}_{\ast}$
and $\widetilde{\nabla}_{\ast}$ (see Definition~\ref{HCT-Def}) to generalize meet-convolution and join-convolution
on fuzzy truth values and proved that the binary operation $\widetilde{\Delta}_{\ast}$ (resp., $\widetilde{\nabla}_{\ast}$) is a t$_{r}$-norm (resp., a t$_{r}$-conorm) on normal and convex fuzzy truth
values under some additional conditions. Then, Wu and Chen~\cite{WC-TFS} answered negatively an open problem
posed by Hern\'{a}ndez et al.~\cite{HCT2015}, proving that the binary operation $\ast$, which ensures
that $\widetilde{\Delta}_{\ast}$ is a t$_{r}$-norm or $\widetilde{\nabla}_{\ast}$ is a t$_{r}$-conorm, is a
t-norm on $I$. Moreover, they~\cite{WCW2022} constructed a t$_{r}$-norm and a t$_{r}$-conorm, which cannot
be obtained by the formulas of `$\widetilde{\Delta}_{\ast}$' and `$\widetilde{\nabla}_{\ast}$'. This answered
another open problem posed by Hern\'{a}ndez et al.~\cite{HCT2015}. Zhang and Hu~\cite{ZH2020-2,ZH2020}
investigated the distributivity laws of convolution operations over meet-convolution and join-convolution
and characterized the idempotency of convolution operations on fuzzy truth values. De Miguel et al.
\cite{DBD2018} proposed two convolution operations on the set of functions between two bounded lattices and
studied their algebraic structure. Liu and Wang \cite{LW2020,LW2020-2} introduced Z-extended overlap functions
and grouping functions on fuzzy truth values according to Zadeh's extension principle and studied the
distributivity laws between the Z-extended overlap functions and grouping functions and between extended
t-norms and t-conorms on fuzzy truth values. Zhang and Wang \cite{ZW2018} obtained a necessary and sufficient
condition ensuring the absorption laws in the algebra of fuzzy truth values. By applying automorphism,
Cubillo et al. \cite{CTH2020} obtained an equivalent characterization for the strong negations on norm and
convex fuzzy truth values with a fixed point $\bm{1}$ ($\bm{1}(x)=1$ for all $x\in I$).

Inspired by the above research progress, in this paper we further study
t-norms for normal and convex fuzzy truth values. We first obtain the implication relations
among three notions of t-norms for fuzzy truth values, proving that
t$_{lor}$-norm on $\mathbf{L}$ (normal and convex fuzzy truth values) is strictly stronger
than t$_{r}$-norm on $\mathbf{L}$, which is strictly stronger than
t-norm on $\mathbf{L}$. Moreover, we characterize the
restrictive axioms (O5), (O5'), and (O6) (see Definition~\ref{Def-1})
for the binary operations $\widetilde{\Delta}_{\ast}$ and $\widetilde{\nabla}_{\ast}$. In particular, we prove
that if the binary operation ${\widetilde{\Delta}_{\ast}}$ is a t$_{r}$-norm (resp., ${\widetilde{\nabla}_{\ast}}$ is a t$_{r}$-conorm) on $\mathbf{L}$,
then ${\Delta}$ is a continuous t-norm (resp., ${\nabla}$ is a continuous t-conorm) on $I$, and
${\ast}$ is a t-norm on $I$, improving the main results in \cite{WC-TFS}.

\section{Preliminaries}\label{S-2}
In this section, we recall some basic concepts and essential
requirements which are used in the sequel.

\subsection{Basic operations and orders on fuzzy truth values}

Throughout this paper, let $Map(X, Y)$ be the set of all mappings from $X$ to $Y$, and `$\leq$'
denote the usual order relation in the lattice of real numbers, $\mathbf{M}=Map(I, I)$.
In the context of type-2 fuzzy sets, the members of $\mathbf{M}$ are sometimes called
{\it fuzzy truth values} or {\it membership grades} of type-2 fuzzy sets.
Let $\vee$ and $\wedge$ be the {\color{blue}maximum and minimum operations, respectively, on the lattice
$\mathbb{R}$ of real numbers}.

\begin{definition}[{\textrm{\protect\cite{WW2005}}}]
A function $f\in \mathbf{M}$ is
\begin{enumerate}[(1)]
  \item {\it normal} if $\sup\{f(x)\mid x\in I\}=1$;
  \item {\it convex} if, for any $0\leq x\leq y\leq z\leq 1$, $f(y)\geq f(x)\wedge f(z)$.
\end{enumerate}
\end{definition}


For any subset $B$ of $X$, a special fuzzy set $\bm{1}_{B}$,
called the {\it characteristic function}
of $B$, is defined by
$$
\bm{1}_{B}(x)=\begin{cases}
1, & x\in B, \\
0, & x\in X\backslash B.
\end{cases}
$$

Let $\mathbf{N}=\{f\in \mathbf{M}\mid f \text{ is normal}\}$,
$\mathbf{C}=\{f\in \mathbf{M}\mid f \text{ is convex}\}$,
$\mathbf{L}=\mathbf{N}\cap \mathbf{C}$,
$\mathbf{J}=\{\bm{1}_{\{x\}}\mid x\in I\}$, and
$\mathbf{K}=\{\bm{1}_{[a, b]}\mid 0\leq a \leq b\leq 1\}$.


The following basic operations on $\mathbf{M}$ were introduced by
Mizumoto and Tanaka~\cite{MT1976}.

\begin{definition}[{\textrm{\protect\cite{WW2005}}}]
\label{Def-7}
The operations of $\sqcap$ ({\it meet-convolution}), $\sqcup$ ({\it join-convolution}), $\neg$ ({\it negation}) on $\mathbf{M}$ are
defined as follows: for $f$, $g\in \mathbf{M}$,
\begin{align*}
(f\sqcap g)(x)&=\sup\{f(y)\wedge g(z)\mid y\wedge z=x\},\\
(f\sqcup g)(x)&=\sup\{f(y)\wedge g(z)\mid y\vee z=x\},
\end{align*}
and
$$
(\neg f)(x)=\sup\{f(y)\mid 1-y=x\}=f(1-x).
$$
\end{definition}

From \cite{WW2005}, it follows that $\mathbb{M}=(\mathbf{M}, \sqcup, \sqcap, \neg, \bm{1}_{\{0\}}, \bm{1}_{\{1\}})$
does not have a lattice structure, although $\sqcup$ and $\sqcap$ satisfy the De Morgan's laws with respect
to $\neg$.

Walker and Walker \cite{WW2005} introduced the following partial orders $\sqsubseteq$ and $\preceq$ on $\mathbf{M}$.
\begin{definition}[{\textrm{\protect\cite{WW2005}}}]
\label{Def-order}
$f\sqsubseteq g$ if $f\sqcap g=f$; $f\preceq g$ if $f\sqcup g=g$.
\end{definition}

It follows from \cite[Proposition 14]{WW2005} that both $\sqsubseteq$ and $\preceq$ are partial
orders on $\mathbf{M}$. Generally, the partial orders $\sqsubseteq$ and $\preceq$ do not coincide.
In \cite{WW2005,HWW2010,HWW2008,MT1976}, it was proved that the partial orders
$\sqsubseteq$ and $\preceq$ coincide on
$\mathbf{L}$, and the subalgebra $\mathbb{L}=(\mathbf{L}, \sqcup, \sqcap, \neg, \bm{1}_{\{0\}}, \bm{1}_{\{1\}})$
is a bounded complete lattice. In particular, $\bm{1}_{\{0\}}$ and $\bm{1}_{\{1\}}$ are the minimum and maximum
of $\mathbb{L}$, respectively. For systematic study on the algebra of fuzzy truth values, one is referred to
\cite{HWW2016,WW2005}.

{\color{blue}\begin{definition}
For $f\in \mathbf{M}$, define 
\begin{align*}
f^{L}(x)&=\sup\{f(y)\mid y\leq x\},\\
f^{L_{\mathrm{w}}}(x)&=\begin{cases}
\sup\{f(y)\mid y< x\}, & x\in (0, 1], \\
f(0), & x=0,
\end{cases}
\end{align*}
and
\begin{align*}
f^{R}(x)&=\sup\{f(y)\mid y\geq x\},\\
f^{R_{\mathrm{w}}}(x)&=\begin{cases}
\sup\{f(y)\mid y> x\}, & x\in [0, 1), \\
f(1), & x=1.
\end{cases}
\end{align*}
\end{definition}
Clearly, (1) $f^L$, $f^{L_{\mathrm{w}}}$ and $f^R$, $f^{R_{\mathrm{w}}}$ are monotonically increasing and
decreasing, respectively; (2) $f^{L}(x)\vee f^{R}(x)=f^{L}(x)\vee f^{R_{\mathrm{w}}}(x)=\sup_{z\in I}\{f(z)\}$ and
$f^{R}(x)\vee f^{L_{\mathrm{w}}}(x)=\sup_{z\in I}\{f(z)\}$ for all $x\in I$. The following properties of $f^{L}$
and $f^{R}$ are obtained by Walker et al.~\cite{HWW2008,HWW2010,WW2005}.

\begin{proposition}[{\textrm{\protect\cite{WW2005}}}]
\label{Proposition-1}
For $f$, $g\in \mathbf{M}$,
\begin{enumerate}[{\rm (1)}]
  \item $f\leq f^{L}\wedge f^{R}$;
  \item $(f^{L})^{L}=f^{L}$, $(f^{R})^{R}=f^{R}$;
  \item $(f^{L})^{R}=(f^{R})^{L}=\sup_{x\in I} \{f(x)\}$;
  \item $f\sqsubseteq g$ if and only if $f^{R}\wedge g \leq f\leq g^{R}$;
  \item $f\preceq g$ if and only if $f\wedge g^{L}\leq g\leq f^{L}$;
  \item $f$ is convex if and only if $f=f^{L}\wedge f^{R}$.
\end{enumerate}
\end{proposition}

\begin{theorem}[{\textrm{\protect\cite{HWW2010,HWW2008}}}]
\label{order-theorem}
Let $f$, $g\in \mathbf{L}$. Then, $f\sqsubseteq g$ if and only if
$g^L\leq f^L$ and $f^R\leq g^R$.
\end{theorem}

\begin{lemma}
[{\textrm{\protect\cite[Corollary~5, Proposition~6]{WW2005}}}]\label{f-g-L-R}
For $f$, $g\in \mathbf{L}$,
\begin{itemize}
\item[{\rm (1)}] $(f\sqcap g)^{L}=f^{L}\vee g^{L}$ and $(f\sqcap g)^{R}=f^{R}\wedge g^{R}$;
\item[{\rm (2)}] $(f\sqcup g)^{L}=f^{L}\wedge g^{L}$ and $(f\sqcup g)^{R}=f^{R}\vee g^{R}$.
\end{itemize}
\end{lemma}

\begin{lemma}[{\textrm{\protect \cite[Theorem~4]{WW2005}}}]
\label{Max-Min-Operation}
Let $f$, $g\in \mathbf{M}$. Then,
\begin{enumerate}[{\rm (1)}]
  \item $f\sqcup g=(f\vee g)\wedge (f^{L}\wedge g^{L})$,
  \item $f\sqcap g=(f\vee g)\wedge (f^{R}\wedge g^{R})$.
\end{enumerate}
\end{lemma}
}

\subsection{T-norms and convolution operations on $\mathbf{M}$}

\begin{definition}[{\textrm{\protect\cite{KMP2000}}}]
A binary operation $\ast: I^{2}\rightarrow I$ is a {\it t-norm} on $I$ if it satisfies the following axioms:
\begin{itemize}
  \item[(T1)] ({\it commutativity}) $x\ast y=y\ast x$ for $x$, $y\in I$;
  \item[(T2)] ({\it associativity}) $(x \ast y) \ast z=x\ast (y\ast z)$ for $x$, $y$, $z\in I$;
  \item[(T3)] ({\it increasing}) $\ast$ is increasing in each argument;
  \item[(T4)] ({\it neutral element}) $1\ast x=x\ast 1=x$ for $x\in I$.
\end{itemize}
A binary operation $\ast: I^2\rightarrow I$ is a
{\it t-conorm} on $I$ if it satisfies axioms (T1), (T2), and (T3) above,
and moreover axiom (T4'): $0\ast x=x\ast 0=x$ for $x\in I$.
\end{definition}

In 2006, Walker and Walker~\cite{WW2006} extended t-norm and t-conorm on $I$ to the algebra
of fuzzy truth values (see Definition~\ref{Def-1} below). Then, Hern\'{a}ndez et al.~\cite{HCT2015}
modified the definition of Walker and Walker, and introduced the notions of t$_{r}$-norm and t$_{r}$-conorm by adding some ``restrictive axioms'' (see Definition~\ref{Def-1} below).

\begin{definition}[{\textrm{\protect\cite{HCT2015,WW2006}}}]
\label{Def-1}
A binary operation $T: \mathbf{L}^2 \rightarrow \mathbf{L}$ is a {\it t$_{r}$-norm}
({\it t-norm according to the restrictive axioms}), if
\begin{itemize}
  \item[(O1)] $T$ is commutative, i.e., $T(f, g)=T(g, f)$ for $f$, $g\in \mathbf{L}$;
  \item[(O2)] $T$ is associative, i.e., $T(T(f, g), h)=T(f, T(g, h))$ for $f$, $g$, $h\in \mathbf{L}$;
  \item[(O3)] $T(f, \bm{1}_{\{1\}})=f$ for $f\in \mathbf{L}$ (neutral element);
  \item[(O4)] for $f, g, h\in \mathbf{L}$ such that $f\sqsubseteq g$,
  $T(f, h)\sqsubseteq T(g, h)$ (increasing in each argument);
  \item[(O5)] $T(\bm{1}_{[0, 1]}, \bm{1}_{[a, b]})=\bm{1}_{[0, b]}$;
  \item[(O6)] $T$ is closed on $\mathbf{J}$;
  \item[(O7)] $T$ is closed on $\mathbf{K}$.
\end{itemize}

A binary operation $S: \mathbf{L}^2\rightarrow \mathbf{L}$ is a
{\it t$_{r}$-conorm} if it satisfies axioms (O1), (O2), (O4), (O6), and (O7) above,
axiom (O3'): $S(f, \bm{1}_{\{0\}})=f$, and axiom (O5'):
$S(\bm{1}_{[0, 1]}, \bm{1}_{[a, b]})=\bm{1}_{[a, 1]}$. Axioms (O1), (O2), (O3), (O3'),
and (O4) are called ``{\it basic axioms}'', and an operation that complies with
these axioms will be referred to as {\it t-norm} and {\it t-conorm}, respectively.
\end{definition}


\begin{definition}[{\textrm{\protect\cite[Definition~5.2.6]{HWW2016}}}]
A binary operation $T: \mathbf{L}^2\rightarrow \mathbf{L}$ is a
{\it lattice-ordered t$_{r}$-norm} ({\it t$_{lor}$-norm} for short) if it satisfies axioms
(O1), (O2), (O3), (O5), (O6), and (O7) above,
axiom (O4'): $T(f, g\sqcup h)=T(f, g)\sqcup T(f, h)$, and axiom (O4"):
$T(f, g\sqcap h)=T(f, g)\sqcap T(f, h)$.
\end{definition}


The following binary operations are introduced by Hern\'{a}ndez et al.~\cite{HCT2015}
for generalizing meet-convolution and join-convolution on $\mathbf{M}$.

\begin{definition}[{\textrm{\protect\cite{HCT2015}}}]\label{HCT-Def}
Let $\ast$ be a binary operation on $I$, $\Delta$ be a
t-norm on $I$, and $\nabla$ be a t-conorm on $I$. Define the binary operations
$\widetilde{\Delta}_{\ast}$ and $\widetilde{\nabla}_{\ast}: \mathbf{M}^2\rightarrow \mathbf{M}$
as follows: for $f$, $g\in \mathbf{M}$,
\begin{equation}\label{O-1}
(f\widetilde{\Delta}_{\ast} g)(x)=\sup\left\{f(y)\ast g(z)\mid y\Delta z =x\right\},
\end{equation}
and
\begin{equation}\label{O-2}
(f\widetilde{\nabla}_{\ast} g)(x)=\sup\left\{f(y)\ast g(z)\mid y\nabla z =x\right\}.
\end{equation}
\end{definition}

\section{\color{blue}Relation between t$_{lor}$-norm and t$_{r}$-norm}\label{S-V}

This section reveals the relation between t$_{lor}$-norm and t$_{r}$-norm on $\mathbf{L}$. In particular,
it is shown that t$_{lor}$-norm is strictly stronger than t$_{r}$-norm on $\mathbf{L}$, by {\color{blue}constructing a
t$_{r}$-norm which is not a t$_{lor}$-norm. This means that every t$_{lor}$-norm on $\mathbf{L}$ is a t$_{r}$-norm on
$\mathbf{L}$, not vice versa.

\begin{theorem}\label{Thm-WW->tr}
Let $T$ be a t$_{lor}$-norm on $\mathbf{L}$. Then, it is a t$_{r}$-norm on $\mathbf{L}$.
\end{theorem}
\begin{proof}
It suffices to check that axiom (O4'') implies axiom (O4). For $f$, $g$, $h\in \mathbf{L}$ with
$f\sqsubseteq g$, it is clear that $T(f, h)\sqcap T(g, h)=T(f\sqcap g, h)=T(f, h)$ by axiom (O4''),
implying that
$T(f, h)\sqsubseteq T(g, h)$ by Definition~\ref{Def-order}.
\end{proof}

It should be noted that, according to \cite[Theorem~5.5.3]{HWW2016}
and \cite[Proposition~14]{HCT2015}, the binary operation $\widetilde{\Delta}_{\min}$ is a t$_{lor}$-norm on $\mathbf{L}$,
provided that $\Delta$ is a continuous t-norm on $I$. In particular, the meet-convolution $\sqcap$ is a t$_{lor}$-norm on
$\mathbf{L}$. }

\begin{lemma}[{\textrm{\protect\cite[Lemma~2.2]{WCW2022}}}]
\label{<-Lemma}
For $f\in \mathbf{N}$, $\inf\{x\in I\mid f^{L}(x)=1\}\leq \sup\{x\in I\mid f^{R}(x)=1\}$.
\end{lemma}

\begin{definition}[{\textrm{\protect\cite[Definition~4.1]{WCW2022}}}]
\label{Def-Star}
Define a binary operation $\hollowstar: \mathbf{L}^2\rightarrow \mathbf{M}$ as follows:
for $f$, $g\in \mathbf{L}$,

(1) $f=\bm{1}_{\{1\}}$, $f\hollowstar g=g\hollowstar f=g$;


(2) $f\neq \bm{1}_{\{1\}}$ and $g\neq \bm{1}_{\{1\}}$,
\begin{equation}\label{xin-operation}
(f\hollowstar g)(t)=\begin{cases}
f^{L}(t)\vee g^{L}(t), & t\in [0, \eta),\\
1, & t\in [\eta, \xi),\\
f^{R}(\xi)\wedge g^{R}(\xi), & t=\xi, \\
0, & t\in (\xi, 1],
\end{cases}
\end{equation}
where $\eta=\inf\{x\in I\mid f^{L}(x)=1\}\wedge \inf\{x\in I\mid g^{L}(x)=1\}$ and
$\xi=\sup\{x\in I\mid f^{R}(x)=1\}\wedge \sup\{x\in I\mid g^{R}(x)=1\}$.
\end{definition}
Clearly, $f\hollowstar g$
is increasing on $[0, \eta)$. Meanwhile, applying Lemma~\ref{<-Lemma} yields that $\eta \leq \xi$.

\begin{theorem}[{\textrm{\protect\cite[Theorem~4.8]{WCW2022}}}]
\label{WW-Thm}
The binary operation $\hollowstar$ is a t$_{r}$-norm on $\mathbf{L}$.
\end{theorem}

\begin{theorem}[{\textrm{\protect\cite[Proposition~4.2]{WCW2022}}}]
\label{LR*-Thm}
For $f$, $g\in \mathbf{L}\backslash \left\{\bm{1}_{\{1\}}\right\}$,
\begin{align}
(f\hollowstar g)^{L}(t)&=
\begin{cases}
f^{L}(t)\vee g^{L}(t), & t\in [0, \eta),\\
1, & t\in [\eta, 1],
\end{cases}\\
(f\hollowstar g)^{R}(t)&=
\begin{cases}
1, & t\in [0, \xi),\\
f^{R}(\xi)\wedge g^{R}(\xi), & t=\xi, \\
0, & t\in (\xi, 1],
\end{cases}
\end{align}
where $\eta=\inf\{x\in I\mid f^{L}(x)=1\}\wedge \inf\{x\in I\mid g^{L}(x)=1\}$, and
$\xi=\sup\{x\in I\mid f^{R}(x)=1\}\wedge\sup\{x\in I\mid g^{R}(x)=1\}$.
\end{theorem}

\begin{proposition}[{\textrm{\protect\cite[Lemma~4.1]{WCW2022}}}]
\label{not=1}
For $f$, $g\in \mathbf{L}\backslash \left\{\bm{1}_{\{1\}}\right\}$, $f\hollowstar g \neq \bm{1}_{\{1\}}$.
\end{proposition}

The following theorem shows that the binary operation $\hollowstar$ satisfies
the distributivity law for meet-convolution $\sqcap$.

\begin{theorem}
For $f$, $g$, $h\in \mathbf{L}$, $f \hollowstar (g\sqcap h)=(f\hollowstar g) \sqcap (f\hollowstar h)$.
\end{theorem}
\begin{proof}
Consider the following two cases:

\medskip

{\color{blue}Case 1. If one of $f$, $g$, and $h$ is equal to $\bm{1}_{\{1\}}$, we consider the following three subcases:

Case 1.1. If $f=\bm{1}_{\{1\}}$, it is clear that
$f \hollowstar (g\sqcap h)=g\sqcap h=(f\hollowstar g) \sqcap (f\hollowstar h)$;

Case 1.2. If $g=\bm{1}_{\{1\}}$, by direct calculation, 
we have $f \hollowstar (g\sqcap h)=f \hollowstar h$ and 
$(f\hollowstar g) \sqcap (f\hollowstar h)=f \sqcap (f\hollowstar h)$. Meanwhile, 
by Theorem~\ref{WW-Thm}, we get 
$f\hollowstar h\sqsubseteq f\hollowstar \bm{1}_{\{1\}}=f$, i.e., 
$f\sqcap (f\hollowstar h)=f\hollowstar h$. And thus $f \hollowstar (g\sqcap h)=f \hollowstar h
=(f\hollowstar g) \sqcap (f\hollowstar h)$;

Case 1.3. If $h=\bm{1}_{\{1\}}$, similarly to Case 1.2, it can be verified that 
$f \hollowstar (g\sqcap h)=f \hollowstar g
=(f\hollowstar g) \sqcap (f\hollowstar h)$.}

\medskip

Case 2. If none of $f$, $g$, and $h$ are equal to $\bm{1}_{\{1\}}$, from Theorem~\ref{LR*-Thm} and
Proposition~\ref{not=1}, it follows that
\begin{align}
(f\hollowstar g)^{L}(t)&=
\begin{cases}
f^{L}(t)\vee g^{L}(t), & t\in [0, \eta_1),\\
1, & t\in [\eta_1, 1],
\end{cases}\label{e-1.1}\\
(f\hollowstar g)^{R}(t)&=
\begin{cases}
1, & t\in [0, \xi_1),\\
f^{R}(\xi_1)\wedge g^{R}(\xi_1), & t=\xi_1, \\
0, & t\in (\xi_1, 1],
\end{cases}\label{e-1.2}
\end{align}
where $\eta_1=\inf\{x\in I\mid f^{L}(x)=1\}\wedge \inf\{x\in I\mid g^{L}(x)=1\}$ and
$\xi_1=\sup\{x\in I\mid f^{R}(x)=1\}\wedge\sup\{x\in I\mid g^{R}(x)=1\}$; and
\begin{align}
(f\hollowstar h)^{L}(t)&=
\begin{cases}
f^{L}(t)\vee h^{L}(t), & t\in [0, \eta_2),\\
1, & t\in [\eta_2, 1],
\end{cases}\label{e-2.1}\\
(f\hollowstar h)^{R}(t)&=
\begin{cases}
1, & t\in [0, \xi_2),\\
f^{R}(\xi_2)\wedge h^{R}(\xi_2), & t=\xi_2, \\
0, & t\in (\xi_2, 1],
\end{cases}\label{e-2.2}
\end{align}
where $\eta_2=\inf\{x\in I\mid f^{L}(x)=1\}\wedge \inf\{x\in I\mid h^{L}(x)=1\}$ and
$\xi_2=\sup\{x\in I\mid f^{R}(x)=1\}\wedge\sup\{x\in I\mid h^{R}(x)=1\}$. According to
Lemma~\ref{f-g-L-R}, we get
\begin{equation}\label{fgh^L}
\begin{split}
&\quad((f\hollowstar g)\sqcap (f\hollowstar h))^{L}(t)\\
&=(f\hollowstar g)^{L}\vee (f\hollowstar h)^{L}\\
&=\begin{cases}
f^{L}(t)\vee g^{L}(t)\vee h^{L}(t), & t\in [0, \eta_1\wedge \eta_{2}),\\
1, & t\in [\eta_1\wedge \eta_2, 1].
\end{cases}
\end{split}
\end{equation}

\begin{claim}\label{claim-1}
$(f\hollowstar g)^{R}(\xi_1\wedge \xi_2)
\wedge (f\hollowstar h)^{R}(\xi_1\wedge \xi_2)=f^{R}(\xi_1\wedge \xi_2)
\wedge g^{R}(\xi_1\wedge \xi_2)\wedge h^{R}(\xi_1\wedge \xi_2)$.
\end{claim}

Case 2-1. If $\xi_1=\xi_2$, then $(f\hollowstar g)^{R}(\xi_1\wedge \xi_2)
\wedge (f\hollowstar h)^{R}(\xi_1\wedge \xi_2)=f^{R}(\xi_1)\wedge g^{R}(\xi_1) \wedge f^{R}(\xi_2)\wedge h^{R}(\xi_2)
=f^{R}(\xi_1)\wedge g^{R}(\xi_1)\wedge h^{R}(\xi_1)$.

Case 2-2. If $\xi_1<\xi_2$, then $(f\hollowstar g)^{R}(\xi_1)=f^{R}(\xi_1)\wedge g^{R}(\xi_1)$ and
$(f\hollowstar h)^{R}(\xi_1)=1$, implying that
\begin{equation}\label{e-1}
(f\hollowstar g)^{R}(\xi_1) \wedge (f\hollowstar h)^{R}(\xi_1)=f^{R}(\xi_1)\wedge g^{R}(\xi_1).
\end{equation}
From $\xi_1<\xi_2\leq \sup\{x\in I\mid h^{R}(x)=1\}$, it follows that
there exists $\xi_1<\hat{x}\leq \sup\{x\in I\mid h^{R}(x)=1\}$ such that $h^{R}(\hat{x})=1$. Thus,
$$
h^{R}(\xi_1)\geq h^{R}(\hat{x})=1.
$$
Together with (\ref{e-1}), we have
$$
(f\hollowstar g)^{R}(\xi_1) \wedge (f\hollowstar h)^{R}(\xi_1)=f^{R}(\xi_1)\wedge g^{R}(\xi_1)\wedge h^{R}(\xi_1).
$$

Case 2-3. If $\xi_2<\xi_1$, similarly to the proof of Case 2-2, it can be verified that
$$
(f\hollowstar g)^{R}(\xi_2) \wedge (f\hollowstar h)^{R}(\xi_2)=f^{R}(\xi_2)\wedge g^{R}(\xi_2)\wedge h^{R}(\xi_2).
$$

Combining Lemma~\ref{f-g-L-R}, (\ref{e-1.2}), (\ref{e-2.2}), and Claim~\ref{claim-1} yields that
\begin{equation}\label{fgh^R}
\begin{split}
&\quad ((f\hollowstar g)\sqcap (f\hollowstar h))^{R}(t)\\
&= (f\hollowstar g)^{R}(t)\wedge (f\hollowstar h)^{R}(t)\\
&= \begin{cases}
1, & t\in [0, \xi_1\wedge \xi_2),\\
f^{R}(t)\wedge g^{R}(t) \wedge h^{R}(t), & t=\xi_1\wedge \xi_2, \\
0, & t\in (\xi_1\wedge \xi_2, 1].
\end{cases}
\end{split}
\end{equation}

\begin{claim}\label{claim-2}
$\inf\{x\in I\mid (g\sqcap h)^{L}(x)=1\}=\inf\{x\in I\mid g^{L}(x)=1\}\wedge \inf\{x\in I\mid
h^{L}(x)=1\}$.
\end{claim}

It is clear that $\{x\in I\mid g^{L}(x)=1\}\cup \{x\in I\mid
h^{L}(x)=1\}=\{x\in I\mid (g^{L}\vee h^{L})(x)=1\}$. Applying
Lemma~\ref{f-g-L-R} yields that
$$
\{x\in I\mid (g\sqcap h)^{L}(x)=1\}=\{x\in I\mid (g^{L}\vee h^{L})(x)=1\}.
$$
Thus,
\begin{align*}
&\quad \inf\{x\in I\mid (g\sqcap h)^{L}(x)=1\}\\
&=\inf(\{x\in I\mid g^{L}(x)=1\}\cup \{x\in I\mid
h^{L}(x)=1\})\\
&=\inf\{x\in I\mid g^{L}(x)=1\}\wedge \inf\{x\in I\mid
h^{L}(x)=1\}.
\end{align*}

\begin{claim}\label{claim-3}
$\sup\{x\in I\mid (g\sqcap h)^{R}(x)=1\}=\sup\{x\in I\mid g^{R}(x)=1\}\wedge \sup\{x\in I\mid
h^{R}(x)=1\}$.
\end{claim}

From Lemma~\ref{f-g-L-R}, it follows that
$\{x\in I\mid (g\sqcap h)^{R}=1\}=\{x\in I\mid (g^{R}\wedge h^{R})(x)=1\}=\{x\in I\mid g^{R}(x)=1\}\cap
\{x\in I\mid h^{R}(x)=1\}$. This implies that
\begin{equation}\label{e-2}
\begin{split}
&\sup\{x\in I\mid (g\sqcap h)^{R}=1\}\\
\leq & \sup\{x\in I\mid g^{R}(x)=1\}\\
& \wedge \sup\{x\in I\mid h^{R}(x)=1\}.
\end{split}
\end{equation}
Since $g^{R}$ and $h^{R}$ are decreasing, we have
$$
{\color{blue}\{x\in I\mid g^{R}(x)=1\} \supseteq \left[0, \sup\{x\in I\mid g^{R}(x)=1\}\right),}
$$
and
$$
{\color{blue}\{x\in I\mid h^{R}(x)=1\} \supseteq \left[0, \sup\{x\in I\mid h^{R}(x)=1\}\right).}
$$
Thus,
{\color{blue}$\{x\in I\mid (g\sqcap h)^{R}=1\} \supseteq [0, \sup\{x\in I\mid g^{R}(x)=1\}\wedge \sup\{x\in I\mid h^{R}(x)=1\})$,}
implying that $\sup\{x\in I\mid (g\sqcap h)^{R}=1\}\geq \sup\{x\in I\mid g^{R}(x)=1\}\wedge \sup\{x\in I\mid  h^{R}(x)=1\}$. By (\ref{e-2}), we get
$\sup\{x\in I\mid (g\sqcap h)^{R}(x)=1\}=\sup\{x\in I\mid g^{R}(x)=1\}\wedge \sup\{x\in I\mid
h^{R}(x)=1\}$.

\medskip

Applying Claims \ref{claim-2}--\ref{claim-3} leads to that
$\inf\{x\in I\mid f^{L}(x)=1\}\wedge \inf\{x\in I\mid (g\sqcap h)^{L}(x)=1\}
=\inf\{x\in I\mid f^{L}(x)=1\}\wedge\inf\{x\in I\mid g^{L}(x)=1\}\wedge \inf\{x\in I\mid
h^{L}(x)=1\}=\eta_1\wedge \eta_2$ and  $\sup\{x\in I\mid f^{R}(x)=1\}\wedge \sup\{x\in I\mid
(g\sqcap h)^{R}(x)=1\}=\sup\{x\in I\mid f^{R}(x)=1\}\wedge\sup\{x\in I\mid g^{R}(x)=1\}\wedge
\sup\{x\in I\mid h^{R}(x)=1\}=\xi_1\wedge \xi_2$. This, together with Theorem~\ref{LR*-Thm}
and Proposition~\ref{not=1}, implies that
\begin{align*}
&\quad(f\hollowstar (g\sqcap h))^{L}(t)\\
&=
\begin{cases}
f^{L}(t)\vee (g\sqcap h)^{L}(t), & t\in [0, \eta_1\wedge \eta_2),\\
1, & t\in [\eta_1\wedge \eta_2, 1],
\end{cases}
\end{align*}
and
\begin{align*}
&\quad(f \hollowstar (g\sqcap h))^{R}(t)\\
&=
\begin{cases}
1, & t\in [0, \xi_1\wedge \xi_2),\\
f^{R}(t)\wedge (g\sqcap h)^{R}(t), & t=\xi_1\wedge \xi_2, \\
0, & t\in (\xi_1\wedge \xi_2, 1].
\end{cases}
\end{align*}
Combining this with Lemma~\ref{f-g-L-R}, (\ref{fgh^L}), and (\ref{fgh^R}) yields that
$$
(f\hollowstar (g\sqcap h))^{L}=((f\hollowstar g)\sqcap (f\hollowstar h))^{L},
$$
and
$$
(f\hollowstar (g\sqcap h))^{R}=((f\hollowstar g)\sqcap (f\hollowstar h))^{R}.
$$
Therefore, {\color{blue}by Proposition~\ref{Proposition-1} (6), we have}
$$
f\hollowstar (g\sqcap h)=(f\hollowstar g)\sqcap (f\hollowstar h).
$$
\end{proof}

The following theorem shows that the binary operation $\hollowstar$
does not satisfy the distributivity laws for meet-convolution $\sqcup$.

\begin{theorem}\label{not-WW-t-norm}
There exist $f$, $g$, $h\in \mathbf{L}$ such that
$f \hollowstar (g\sqcup h)\neq (f\hollowstar g) \sqcup (f\hollowstar h)$.
\end{theorem}
\begin{proof}
Choose respectively $f$, $g$, $h\in \mathbf{L}$ as follows:
\begin{align*}
f(x)&=\bm{1}_{\{0.75\}}(x), \ x\in [0, 1], \\
g(x)&=
\begin{cases}
0, & x\in [0, 0.5], \\
2(1-x), & x\in (0.5, 1],
\end{cases}\\
h(x)&=x, \ x\in [0, 1].
\end{align*}
\begin{enumerate}[(i)]
\item It can be verified that
\begin{align*}
f^{L}(x)&=\begin{cases}
0, & x\in [0, 0.75), \\
1, & x\in [0.75, 1],
\end{cases}\\
f^{R}(x)&=\begin{cases}
1, & x\in [0, 0.75], \\
0, & x\in (0.75, 1],
\end{cases}\\
h^{L}(x)&=x, \ x\in [0, 1], \\
h^{R}&\equiv 1, \\
g^{L}(x)& =\begin{cases}
0, & x\in [0, 0.5],\\
1, & x\in (0.5, 1],
\end{cases}\\
g^{R}&=\begin{cases}
1, & x\in [0, 0.5], \\
2(1-x), & x\in (0.5, 1],
\end{cases}\\
(g\vee h)(x)&=\begin{cases}
x, & x\in [0, 0.5], \\
2(1-x), & x\in [0.5, \frac{2}{3}], \\
x, & x\in (\frac{2}{3}, 1].
\end{cases}
\end{align*}
According to (\ref{xin-operation}), we have
\begin{align*}
(f\hollowstar g)(x)&=\begin{cases}
0, & x\in [0, 0.5), \\
1, & x=0.5,\\
0, & x\in (0.5, 1],
\end{cases}\\
(f\hollowstar h)(x)&=\begin{cases}
x, & x\in [0, 0.75), \\
1, & x=0.75, \\
0, & x\in (0.75, 1].
\end{cases}
\end{align*}
Combining this with Lemma~\ref{Max-Min-Operation}, it follows that
\begin{equation}\label{e-3.2}
((f\hollowstar g) \sqcup (f\hollowstar h))(0.5)=(1\vee 0.5)\wedge (1\wedge 0.5)
=0.5.
\end{equation}
\item From Lemma~\ref{Max-Min-Operation}, it follows that
\begin{equation}\label{e-3.1}
\begin{split}
&\quad (g\sqcup h)(x)\\
&=((g\vee h) \wedge (g^{L}\wedge h^{L}))(x)\\
&=\begin{cases}
0, & x\in [0, 0.5], \\
x, & x\in (0.5, 1].
\end{cases}
\end{split}
\end{equation}
Then,
\begin{align*}
(g\sqcup h)^{L}(x)&=\begin{cases}
0, & x\in [0, 0.5], \\
x, & x\in (0.5, 1],
\end{cases}\\
(g\sqcup h)^{R}(x)&\equiv 1.
\end{align*}

\item From (\ref{xin-operation}) and (\ref{e-3.1}), it follows that
\begin{align*}
(f\hollowstar (g\sqcup h))(x)=\begin{cases}
0, & x\in [0, 0.5], \\
x, & x\in (0.5, 0.75), \\
1, & x=0.75,\\
0, & x\in (0.75, 1].
\end{cases}
\end{align*}
In particular, $(f\hollowstar (g\sqcup h))(0.5)=0$. This, together with (\ref{e-3.2}), implies that
$$
f\hollowstar (g\sqcup h)\neq (f\hollowstar g) \sqcup (f\hollowstar h).
$$
\end{enumerate}
\end{proof}

\begin{remark}
From Theorems~\ref{WW-Thm} and \ref{not-WW-t-norm}, it follows that $\hollowstar$ is a
t$_{r}$-norm but not a t$_{lor}$-norm on $\mathbf{L}$. This implies that t$_{lor}$-norm
is strictly stronger that t$_{r}$-norm by Theorem~\ref{Thm-WW->tr}.
\end{remark}

\section{\color{blue} Relation between T$_{r}$-norm and t-norm}

Clearly, t$_{r}$-norm is stronger than t-norm, i.e., every t$_{r}$-norm
on $\mathbf{L}$ is a t-norm on $\mathbf{L}$. This section gives an example to show that this is strict.

\begin{definition}\label{bingstar-operation-2}
Define a binary operation $\chinastar: \mathbf{L}^2\rightarrow \mathbf{M}$ as follows: for $f, g\in \mathbf{L}$,

(1) $f=\bm{1}_{\{1\}}$, $f\chinastar g=g\chinastar f=g$;


(2) $f\neq \bm{1}_{\{1\}}$ and $g\neq \bm{1}_{\{1\}}$,
\begin{equation}\label{xin-operation-2}
(f\chinastar g)(t)=\begin{cases}
1, & t\in [0, \xi),\\
f^{R}(\xi)\wedge g^{R}(\xi), & t=\xi, \\
0, & t\in (\xi, 1],
\end{cases}
\end{equation}
where $\xi=\sup\{x\in I\mid f^{R}(x)=1\}\wedge \sup\{x\in I\mid g^{R}(x)=1\}$.
\end{definition}
\begin{proposition}
$\chinastar(\mathbf{L}^{2})\subseteq \mathbf{L}$.
\end{proposition}

\begin{theorem}\label{LR**-Thm}
For $f, g\in \mathbf{L}\backslash \left\{\bm{1}_{\{1\}}\right\}$,
\begin{align}
(f\chinastar g)^{L}(t)&\equiv 1, \\
(f\chinastar g)^{R}(t)&=
\begin{cases}
1, & t\in [0, \xi),\\
f^{R}(\xi)\wedge g^{R}(\xi), & t=\xi, \\
0, & t\in (\xi, 1],
\end{cases}
\end{align}
where $\xi=\sup\{x\in I\mid f^{R}(x)=1\}\wedge\sup\{x\in I\mid g^{R}(x)=1\}$.
\end{theorem}

\begin{proposition}\label{not-tr-norm}
The binary operation $\chinastar$ does not satisfy (O6).
\end{proposition}
\begin{proof}
For $x_1$, $x_2\in (0, 1)$, from \eqref{xin-operation-2}, it follows that
\begin{align*}
&\quad (\bm{1}_{\{x_1\}}\chinastar \bm{1}_{\{x_2\}})(t)\\
&=
\begin{cases}
1, & x\in [0, x_1\wedge x_2], \\
0, & x\in (x_1\wedge x_2, 1],
\end{cases}\\
&=\bm{1}_{[0, x_1\wedge x_2]}\notin \mathbf{J}.
\end{align*}
This implies that $\chinastar$ does not satisfy (O6).
\end{proof}

Similarly to the proofs of Theorems 4.1--4.4 in \cite{WCW2022}, applying Theorem~\ref{LR**-Thm} and
Proposition~\ref{not-tr-norm} leads to the following result.
\begin{theorem}\label{t-norm}
The binary operation $\chinastar$ is a t-norm but not a t$_{r}$-norm on $\mathbf{L}$.
\end{theorem}

\section{Some further results on the binary operations $\widetilde{\Delta}_{\ast}$
and $\widetilde{\nabla}_{\ast}$}

The following lemma, originated from \cite[Proposition~1.19]{KMP2000},
shows that the continuity in one component is sufficient for the continuity
of t-norms.

\begin{lemma}[{\textrm{\protect\cite[Proposition~1.19]{KMP2000}}}]\label{continuous-lemma}
A binary operation $T: I^2\rightarrow I$, which satisfies (T3), is continuous
if and only if it is continuous in each component, i.e., for all $x_0$, $y_0\in I$,
both the vertical section $T(x_0, \_): I\rightarrow I$ and the horizontal section
$T(\_, y_0): I\rightarrow I$ are continuous functions in one variable.
\end{lemma}

\begin{proposition}
[{\textrm{\protect\cite[Lemma~3.1]{WCW2022}}}]
\label{1-Lemma}
\begin{enumerate}[{\rm (1)}]
\item Let $\ast$ be a t-norm on $I$. Then, $x\ast y=1$ if and only if $x=y=1$.
\item Let $\ast$ be a t-conorm on $I$. Then, $x\ast y=0$ if and only if $x=y=0$.
\end{enumerate}
\end{proposition}

\begin{proposition}\label{product}
Let $\Delta$ be a t-norm on $I$ and $\ast$ be a binary operation on $I$ satisfying that
$0\ast 0=0\ast 1=1\ast 0=0$. If $\widetilde{\Delta}_{\ast}$ satisfies axiom (O6), then, for
$x_1$, $x_2\in I$, we have $1\ast 1=1$ and $\bm{1}_{\{x_1\}}\widetilde{\Delta}_{\ast}
\bm{1}_{\{x_2\}}=\bm{1}_{\{x_1\Delta x_2\}}$.
\end{proposition}

\begin{proof}
From $0\ast 0=0\ast 1=1\ast 0=0$, it follows that
\begin{itemize}
  \item[(a)] for $y$, $z\in I$, $\bm{1}_{\{x_1\}}(y)\ast \bm{1}_{\{x_2\}}(z)\in\{0, 1\ast 1\}$;
  \item[(b)] $\bm{1}_{\{x_1\}}(y)\ast \bm{1}_{\{x_2\}}(z)=1\ast 1$ if and only if $y=x_1$ and $z=x_2$.
\end{itemize}
By $(\bm{1}_{\{x_1\}}\widetilde{\Delta}_{\ast} \bm{1}_{\{x_2\}})(x)
=\sup\{\bm{1}_{\{x_1\}}(y)\ast \bm{1}_{\{x_2\}}(z)\mid y\Delta z=x\}$, we have
$$
(\bm{1}_{\{x_1\}}\widetilde{\Delta}_{\ast} \bm{1}_{\{x_2\}})(x)=
\begin{cases}
1\ast 1, & x=x_1\Delta x_2, \\
0, & x\in I\backslash\{x_1\Delta x_2\}.
\end{cases}
$$
Since $\widetilde{\Delta}_{\ast}$ satisfies axiom (O6), we get $1\ast 1=1$ and
$\bm{1}_{\{x_1\}}\widetilde{\Delta}_{\ast} \bm{1}_{\{x_2\}}=\bm{1}_{\{x_1\Delta x_2\}}.$
\end{proof}

\begin{theorem}\label{O5}
Let $\ast$ be a binary operation on $I$ and $\Delta$ be a t-norm on $I$.
Then, the following statements are equivalent:
\begin{enumerate}[{\rm (1)}]
  \item $\widetilde{\Delta}_{\ast}$ satisfies axiom (O5);
  \item $1\ast 0=0$, $1\ast 1=1$, and the function $\Delta$$(\_, b)$ is continuous for all $b\in I$;
  \item $1\ast 0=0$, $1\ast 1=1$, and $\Delta$ is a continuous t-norm.
\end{enumerate}
\end{theorem}
\begin{proof}
(1) $\Longrightarrow$ (2).

\medskip

(a) $1\ast 0=0$.

\medskip

  Suppose, on the contrary, that $1\ast 0>0$, and fix a closed interval $[0, 0.5]\subseteq I$.
  For $z\in [0, 1]\backslash [0, 0.5]=(0.5, 1]$, since $\widetilde{\Delta}_{\ast}$ satisfies axiom (O5)
  and $1\Delta z=z$, it follows from the definition of $\widetilde{\Delta}_{\ast}$ that
  \begin{align*}
  \bm{1}_{[0, 0.5]}(z)
  &=(\bm{1}_{[0, 1]}\widetilde{\Delta}_{\ast} \bm{1}_{[0, 0.5]})(z)\\
  &\geq \bm{1}_{[0, 1]}(1)\ast \bm{1}_{[0, 0.5]}(z)\\
  &=1\ast 0>0,
  \end{align*}
  implying that $\bm{1}_{[0, 0.5]}(z)=1$. This means that $z\in [0, 0.5]$,
  which contradicts with $z\in (0.5, 1]$.

  \medskip

  (b) $1\ast 1=1$.

\medskip

  Clearly, $1=\bm{1}_{[0, 0.5]}(0)=
  (\bm{1}_{[0, 1]}\widetilde{\Delta}_{\ast} \bm{1}_{[0, 0.5]})(0)\in \{1\ast 0, 1\ast 1\}=\{0, 1\ast 1\}$. This implies that $1\ast 1=1$.

  \medskip

  (c) $\Delta$$(\_, b)$ is continuous for all $b\in I$.

  \medskip

  Clearly, both $\Delta$$(\_, 0)$ and $\Delta$$(\_, 1)$ are continuous.
  For $b\in I$, since $\bm{1}_{[0, 1]}(y)=1$ and $\bm{1}_{\{b\}}(z)\in \{0, 1\}$
  for $y$, $z\in I$, it follows from (b) that, for $y\in I$,
  \begin{equation}\label{32-1}
  (\bm{1}_{[0, 1]}\widetilde{\Delta}_{\ast} \bm{1}_{[b, b]})(y\Delta b)=1.
  \end{equation}
  For $x\in I\backslash \{y\Delta b\mid y\in I\}$, if $y\Delta z=x$, then $z\neq b$.
  This, together with (a), implies that
  \begin{equation}\label{32-2}
  (\bm{1}_{[0, 1]}\widetilde{\Delta}_{\ast} \bm{1}_{[b, b]})(x)
  =\sup\{\bm{1}_{[0, 1]}(y)\ast \bm{1}_{[b, b]}(z)\mid
  y\Delta z=x\}=0.
  \end{equation}
  Since $\widetilde{\Delta}_{\ast}$ satisfies axiom (O5), applying \eqref{32-1} and \eqref{32-2}
  yields that
  $$
  \bm{1}_{[0, b]}=\bm{1}_{[0, 1]}\widetilde{\Delta}_{\ast} \bm{1}_{[b, b]}=\bm{1}_{\{y\Delta z\mid y\in I, z=b\}},
  $$
  i.e.,
  \begin{equation}\label{Ib}
  [0, b]=\{y\Delta z\mid y\in I, z=b\}=I\Delta b.
  \end{equation}

\begin{claim}\label{claim-4}
$\Delta$$(\_, b)$ is continuous for all $b\in (0, 1)$.
\end{claim}

  Suppose, on the contrary, that $\Delta$$(\_, b)$ is not continuous for some $b\in (0, 1)$.
  Then, there exists some $z\in (0, 1)$ such that $\Delta$$(\_, b)$ is not continuous at $z$.
  Consider the following three cases:

  Case 1. $z=0$. Since $\Delta$$(\_, b)$ is increasing, the right-limit of
  $\Delta$$(\_, b)$ at $0$ exists. Let $\xi=\lim_{x\to 0^+}$$\Delta$$(x, b)$. Clearly,
  $\xi>$$\Delta$$(0, b)=0$, since $\Delta$$(x, b)$ is not continuous at $0$.
  Since $\Delta$$(\_, b)$ is increasing, according to~\eqref{Ib}, we get
  \begin{align*}
  &\quad [0, b]=I\Delta b\\
  &= \{0\Delta b\}\cup \{y\Delta b\mid y\in (0, 1]\}\\
  &\subseteq \{0\} \cup [\xi, 1\Delta b]\\
  &=\{0\} \cup [\xi, b]\\
  &\subsetneqq [0, b],
  \end{align*}
  which is a contradiction.

  Case 2. $z\in (0, 1)$. Since $\Delta$$(\_, b)$ is increasing, both the left-limit and the
  right-limit of $\Delta$$(\_, b)$ at $z$ exist. Let $\eta=\lim_{x\to z^-}$$\Delta$$(x, b)$
  and $\xi=\lim_{x\to z^+}$$\Delta$$(x, b)$. Clearly, $\eta<\xi$, since $\Delta$$(x, b)$ is
  not continuous at $z$. This, together with (\ref{Ib}) and the fact that $\Delta$$(\_, b)$ is
  increasing, implies that
  \begin{align*}
  &\quad [0, b]=I\Delta b\\
  &= \{y\Delta b\mid y\in [0, z)\}\cup \{z\Delta b\}\cup \{y\Delta b\mid y\in (z, 1]\}\\
  &\subseteq [0 \Delta b, \eta] \cup \{z\Delta b\}
  \cup [\xi, 1\Delta b]\\
  &=[0, \eta] \cup \{z\Delta b\}\cup [\xi, b]\\
  &\subsetneqq [0, b],
  \end{align*}
which is a contradiction.

Case 3. $z=1$. Similarly to the proof of Case 1, it can be verified that this is true.

\medskip

(2) $\Longrightarrow$ (3). Since $\Delta$ is commutative and satisfies (T3), by applying
Lemma~\ref{continuous-lemma}, this can be verified immediately.

\medskip

(3) $\Longrightarrow$ (1). For a closed interval $[a, b]\subseteq I$, similarly to the proofs
of \eqref{32-1} and \eqref{32-2}, it follows from $1\ast 0=0$ and $1\ast 1=1$ that
\begin{equation}\label{[0,1]*[a,b]}
\bm{1}_{[0, 1]}\widetilde{\Delta}_{\ast} \bm{1}_{[a, b]}=\bm{1}_{\{y\Delta z\mid y\in [0, 1], z\in [a, b]\}}
=\bm{1}_{I\Delta [a, b]}.
\end{equation}
Since $\Delta$ is a continuous t-norm, it can be verified that $[0, b]=I\Delta b\subseteq I\Delta [a, b]
\subseteq [0\Delta a, 1\Delta b]=[0, b].$ This, together with (\ref{[0,1]*[a,b]}), implies that
$\bm{1}_{[0, 1]}\widetilde{\Delta}_{\ast} \bm{1}_{[a, b]}=\bm{1}_{[0, b]}.$
\end{proof}

\begin{theorem}\label{O6}
Let $\ast$ be a binary operation on $I$ and $\Delta$ be a t-norm on $I$.
Then, the following statements are equivalent:
\begin{enumerate}[{\rm (1)}]
  \item $\widetilde{\Delta}_{\ast}$ satisfies axiom (O6);
  \item $0\ast 0=0\ast 1=1\ast 0=0$, $1\ast 1=1$;
  \item for $x_1$, $x_2\in I$, $\bm{1}_{\{x_1\}}\widetilde{\Delta}_{\ast}\bm{1}_{\{x_2\}}=\bm{1}_{\{x_1\Delta x_2\}}$.
\end{enumerate}
\end{theorem}
\begin{proof}
(2) $\Longrightarrow$ (3). For $x_1$, $x_2\in I$, from (2), it follows that
$\{\bm{1}_{\{x_1\}}(y), \bm{1}_{x_2}(z)\mid y, z\in I\}
\subseteq \{0, 1\}$. This, together with (2), implies that
\begin{align*}
&\quad(\bm{1}_{\{x_1\}}\widetilde{\Delta}_{\ast} \bm{1}_{\{x_2\}})(x)\\
&=\sup\{\bm{1}_{\{x_1\}}(y)\ast
\bm{1}_{\{x_2\}}(z)\mid y\Delta z=x\}\\
&=\begin{cases}
0, & x\neq x_1\Delta x_2, \\
1, & x=x_1\Delta x_2,
\end{cases}
\end{align*}
i.e.,
$$
\bm{1}_{\{x_1\}}\widetilde{\Delta}_{\ast} \bm{1}_{\{x_2\}}=
\bm{1}_{\{x_1\Delta x_2\}}.
$$

(3) $\Longrightarrow$ (1). This holds trivially.

(1) $\Longrightarrow$ (2). Fix $x_1\in I$. It is clear that $\bm{1}_{\{x_1\}}\widetilde{\Delta}_{\ast} \bm{1}_{\{0.5\}}\in \mathbf{J}$.

\begin{claim}\label{claim-5}
$0\ast 0=0$.
\end{claim}

  Suppose, on the contrary, that $0\ast 0>0$, and fix $x_1\in I$. From $\bm{1}_{\{x_1\}}\widetilde{\Delta}_{\ast} \bm{1}_{\{0.5\}}\in \mathbf{J}$,
  it follows that there exists $a\in I$ such that $\bm{1}_{\{a\}}=\bm{1}_{\{x_1\}}\widetilde{\Delta}_{\ast} \bm{1}_{\{0.5\}}$.
  This implies that, for $x\in I\backslash\{a, x_1\}$,
  \begin{align*}
  0=&\bm{1}_{\{a\}}(x)=(\bm{1}_{\{x_1\}}\widetilde{\Delta}_{\ast} \bm{1}_{\{0.5\}})(x)\\
  =&\sup\{\bm{1}_{\{x_1\}}(y)\ast \bm{1}_{\{0.5\}}(z)\mid y\Delta z=x\}\\
  \geq & \bm{1}_{\{x_1\}}(x)\ast \bm{1}_{\{0.5\}}(1) \ (\text{by } x\Delta 1=x)\\
  =& 0\ast 0>0,
  \end{align*}
which is a contradiction. 

\begin{claim}\label{claim-6}
$0\ast 1=0$.
\end{claim}

Suppose, on the contrary, that $0\ast 1>0$, and fix $x_{1}\in I$. From
$\bm{1}_{\{x_1\}}\widetilde{\Delta}_{\ast} \bm{1}_{\{1\}}\in \mathbf{J}$, it follows that there
exists $a\in I$ such that $\bm{1}_{\{a\}}=\bm{1}_{\{x_1\}}\widetilde{\Delta}_{\ast} \bm{1}_{\{1\}}$.
This implies that, for $x\in I\backslash \{a, x_1\}$,
\begin{align*}
0&=\bm{1}_{\{a\}}(x)=(\bm{1}_{\{x_1\}}\widetilde{\Delta}_{\ast} \bm{1}_{\{1\}})(x)\\
&=\sup\{\bm{1}_{\{x_1\}}(y)\ast \bm{1}_{\{1\}}(z)\mid y\Delta z=x\}\\
&\geq \bm{1}_{\{x_1\}}(x)\ast \bm{1}_{\{1\}}(1) \ (\text{by } x\Delta 1=x)\\
&=0\ast 1>0,
\end{align*}
which is a contradiction.

Similarly, the following claims can be verified.

\begin{claim}\label{claim-7}
$1\ast 0=0$.
\end{claim}

\begin{claim}\label{claim-8}
$1\ast 1=1$.
\end{claim}

Applying Proposition~\ref{product} and Claims \ref{claim-5}--\ref{claim-8}, this holds trivially.
\end{proof}

\begin{corollary}\label{O5-O6}
Let $\ast$ be a binary operation on $I$ and $\Delta$ be a t-norm on $I$.
If $\widetilde{\Delta}_{\ast}$ satisfies axioms (O5) and (O6), then it satisfies axiom (O7).
\end{corollary}
\begin{proof}
Take two closed intervals $[a_1, b_1]$, $[a_2, b_2]\subseteq I$.
Applying Theorems~\ref{O5} and \ref{O6}, it can be verified that
$$
\bm{1}_{[a_1, b_1]}\widetilde{\Delta}_{\ast} \bm{1}_{[a_2, b_2]}
=\bm{1}_{\left\{x\Delta y\mid x\in [a_1, b_1], y\in [a_2, b_2]\right\}}
=\bm{1}_{[a_1, b_1]\Delta [a_2, b_2]}.
$$
Noting that $[a_1, b_1]$$\times$$[a_2, b_2]$ is a compact and connected subset of $\mathbb{R}^2$
and $\Delta$ is continuous, we know that $[a_1, b_1]\Delta [a_2, b_2]$ is a compact and connected
subset of $\mathbb{R}$. Since $\Delta$ is increasing in each argument, we have
$$
[a_1, b_1]\Delta [a_2, b_2]=[a_1\Delta a_2, b_1\Delta b_2].
$$
Therefore,
$$
\bm{1}_{[a_1, b_1]}\widetilde{\Delta}_{\ast} \bm{1}_{[a_2, b_2]}=\bm{1}_{[a_1\Delta a_2, b_1\Delta b_2]}\in \mathbf{K}.
$$
\end{proof}

\begin{lemma}[{\textrm{\protect \cite[Theorem~2]{WC-TFS}}}]
\label{Thm-20}
Let $\Delta$ be a continuous t-norm on $I$ and $\ast$ be a binary operation on $I$.
If $\widetilde{\Delta}_{\ast}$ is a t$_{r}$-norm on $\mathbf{L}$, then $\ast$ is a t-norm.
\end{lemma}

\begin{lemma}[{\textrm{\protect \cite[Theorem~3]{WC-TFS}}}]
\label{Thm-21}
Let $\nabla$ be a continuous t-conorm on $I$ and $\ast$ be a binary operation on $I$.
If $\widetilde{\nabla}_{\ast}$ is a t$_{r}$-conorm on $\mathbf{L}$, then $\ast$ is a t-norm.
\end{lemma}

\begin{theorem}\label{Tr-Norm-Norm}
Let $\ast$ be a binary operation on $I$ and $\Delta$ be a t-norm on $I$.
If the binary operation $\widetilde{\Delta}_{\ast}$ is a t$_{r}$-norm on $\mathbf{L}$, then
$\Delta$ is a continuous t-norm and $\ast$ is a t-norm.
\end{theorem}
\begin{proof}
Since $\widetilde{\Delta}_{\ast}$ is a t$_{r}$-norm on $\mathbf{L}$, from Theorem~\ref{O5},
it follows that $\Delta$ is a continuous t-norm. This, together with Lemma~\ref{Thm-20}, implies that
$\ast$ is a t-norm on $I$.
\end{proof}

Applying Lemma~\ref{Thm-21},
some slight changes in the proofs of Proposition~\ref{product},
Theorems~\ref{O5}, \ref{O6}, and \ref{Tr-Norm-Norm},
and Corollary~\ref{O5-O6} lead to the following.

\begin{proposition}\label{product-2}
Let $\nabla$ be a t-conorm on $I$ and $\ast$ be a binary operation on $I$ satisfying that
$0\ast 0=0\ast 1=1\ast 0=0$. If $\widetilde{\nabla}_{\ast}$ satisfies axiom (O6), then, for $x_1$, $x_2\in I$,
we have $1\ast 1=1$ and $\bm{1}_{\{x_1\}}\widetilde{\nabla}_{\ast} \bm{1}_{\{x_2\}}=\bm{1}_{\{x_1\nabla x_2\}}$.
\end{proposition}

\begin{theorem}\label{O5'}
Let $\ast$ be a binary operation on $I$ and $\nabla$ be a t-conorm on $I$.
Then, the following statements are equivalent:
\begin{enumerate}[{\rm (1)}]
  \item $\widetilde{\nabla}_{\ast}$ satisfies axiom (O5');
  \item $1\ast 0=0$, $1\ast 1=1$, and the function $\nabla$$(\_, b)$ is continuous for all $b\in I$;
  \item $1\ast 0=0$, $1\ast 1=1$, and $\nabla$ is a continuous t-conorm.
\end{enumerate}
\end{theorem}

\begin{theorem}\label{O6'}
Let $\ast$ be a binary operation on $I$ and $\nabla$ be a t-conorm on $I$.
Then, the following statements are equivalent:
\begin{enumerate}[{\rm (1)}]
  \item $\widetilde{\nabla}_{\ast}$ satisfies axiom (O6);
  \item $0\ast 0=0\ast 1=1\ast 0=0$, $1\ast 1=1$;
  \item for $x_1$, $x_2\in I$, $\bm{1}_{\{x_1\}}\widetilde{\nabla}_{\ast}\bm{1}_{\{x_2\}}=\bm{1}_{\{x_1\nabla x_2\}}$.
\end{enumerate}
\end{theorem}

\begin{corollary}\label{O5'-O6}
Let $\ast$ be a binary operation on $I$ and $\nabla$ be a t-conorm on $I$.
If $\widetilde{\nabla}_{\ast}$ satisfies axioms (O5') and (O6), then it satisfies axiom (O7).
\end{corollary}

\begin{theorem}\label{Tr-Conorm-Conorm}
Let $\ast$ be a binary operation on $I$ and $\nabla$ be a t-conorm on $I$.
If the binary operation $\widetilde{\nabla}_{\ast}$ is a t$_{r}$-conorm on $\mathbf{L}$, then
$\nabla$ is a continuous t-conorm and $\ast$ is a t-norm.
\end{theorem}

\section{Conclusion}
The notions of t-norm, t$_{r}$-norm, and t$_{lor}$-norm on $\mathbf{L}$ were
introduced by Walker and Walker~\cite{WW2006} in 2006 and modified by Hern\'{a}ndez et al.~\cite{HCT2015}.
This paper first proves that the non-convolution binary operation $\hollowstar$ constructed by Wu
et al.~\cite{WCW2022} satisfies the distributivity law for meet-convolution and shows that t$_{lor}$-norm on
$\mathbf{L}$ is strictly stronger than t$_{r}$-norm on $\mathbf{L}$, which is strictly stronger than
t-norm on $\mathbf{L}$. Moreover, some equivalent conditions on restrictive axioms (O5), (O5'), and (O6)
for the binary operations $\widetilde{\Delta}_{\ast}$ and $\widetilde{\nabla}_{\ast}$ are obtained.
As corollaries, it is proved that
\begin{itemize}
\item[(1)] For the binary operation $\widetilde{\Delta}_{\ast}$, axioms (O5) and (O6) imply axiom (O7);
\item[(2)] For the binary operation $\widetilde{\nabla}_{\ast}$, axioms (O5') and (O6) imply axiom (O7).
\end{itemize}
Furthermore, it is proved that, if the binary operation ${\widetilde{\Delta}_{\ast}}$ is a t$_{r}$-norm (resp., ${\widetilde{\nabla}_{\ast}}$
is a t$_{r}$-conorm) on $\mathbf{L}$, then ${\Delta}$ is a continuous t-norm (resp.,
${\nabla}$ is a continuous t-conorm) on $I$, and ${\ast}$ is a t-norm on $I$.

\end{document}